\documentclass[12pt]{article}
\usepackage[utf8]{inputenc}

\usepackage{color}

\usepackage{epsfig}
\usepackage{cite}
\usepackage{latexsym,amsmath}
\usepackage{amsfonts}
\usepackage{float}
\usepackage{pict2e}
\usepackage{tikz}
\usepackage{caption}
\usepackage{amsfonts}
\usepackage{amssymb}
\usepackage{mathrsfs}
\usepackage{amsmath}
\usepackage{enumerate}
\usepackage{graphicx}
\usepackage{subfig}
\usepackage{MnSymbol}
\usepackage{mathtools}
\begin{document}

	\bibliographystyle{plain}
		
		\pagestyle{myheadings}
		\thispagestyle{empty}
		\newtheorem{theorem}{Theorem}
		\newtheorem{corollary}[theorem]{Corollary}
		\newtheorem{definition}{Definition}
		\newtheorem{guess}{Conjecture}
		\newtheorem{problem}{Problem}
		\newtheorem{question}{Question}
		\newtheorem{lemma}[theorem]{Lemma}
		\newtheorem{proposition}[theorem]{Proposition}
		\newtheorem{observation}[theorem]{Observation}
		\newenvironment{proof}{\noindent {\bf
				Proof.}}{\hfill\rule{3mm}{3mm}\par\medskip}
		\newcommand{\remark}{\medskip\par\noindent {\bf Remark.~~}}
		\newcommand{\pp}{{\it p.}}
		\newcommand{\de}{\em}
		\newtheorem{example}{Example}

\title{\bf Chromatic  $\lambda$-choosable and $\lambda$-paintable graphs}

\author{Jialu Zhu\thanks{Department of Mathematics, Zhejiang Normal University, Email: 709747529@qq.com, }         \and
	Xuding Zhu\thanks{Department of Mathematics, Zhejiang Normal University, Email: xdzhu@zjnu.edu.cn, Grant numbers: NSFC 11571319, }}

\maketitle


\begin{abstract}
  Let $\phi(k)$ be the minimum number of vertices in a non-$k$-choosable $k$-chromatic graph. The Ohba conjecture, confirmed by Noel, Reed and Wu, asserts that $\phi(k) \ge 2k+2$.
  This bound is tight if $k$ is even.
  If $k$ is odd, then it is known that $\phi(k) \le 2k+3$ and it is conjectured by Noel 
  that $\phi(k) = 2k+3$. For a multi-set $\lambda=\{k_1,k_2, \ldots, k_q\}$ of positive integers,
  let $k_{\lambda} = \sum_{i=1}^q k_i$.
  A $\lambda$-list assignment of $G$ is a $k_{\lambda}$-list assignment $L$ for which the colour set $\cup_{v \in V(G)}L(v)$ can be partitioned into the disjoint union $C_1 \cup C_2 \cup \ldots \cup C_q$ of $q$ sets so that for each $i$ and each vertex $v$ of $G$, $|L(v) \cap C_i| \ge k_i$. We say $G$ is $\lambda$-choosable if $G$ is $L$-colourable for any $\lambda$-list assignment $L$ of $G$.   
  Let $\phi(\lambda )$ be the minimum number of vertices in a non-$\lambda$-choosable $k_{\lambda}$-chromatic graph.
  Let $1_{\lambda}$ be the multiplicity of $1$ in $\lambda$, and let $o_{\lambda}$ be the number of elements in $\lambda$ that are odd integers. If $1_{\lambda} = k_{\lambda}$ (i.e., $\lambda$ consists of $k_{\lambda}$ copies of $1$'s), then $\lambda$-choosable is equivalent to $k_{\lambda}$-colourable. In this case, $\phi(\lambda) = \infty$. 
  We prove that if  $1_{\lambda} \ne k_{\lambda}$, then $2k_{\lambda}+1_{\lambda}+2 \leqslant \phi(\lambda ) \leqslant 2k_\lambda+ o_\lambda +2$.
  In particular, if $1_{\lambda}=o_{\lambda}=t$, i.e. $\lambda$ contains no odd integer greater than $1$, then $\phi(\lambda ) = 2k_{\lambda}+t+2$.
  We also prove that  $\phi(\lambda) \leqslant 2k_{\lambda}+5 1_{\lambda}+3$. In particular, if $1_{\lambda}=0$, then $2k_{\lambda}+2 \leqslant \phi(\lambda) \leqslant 2k_{\lambda}+3$. We then introduce the concept of $\lambda$-paintability of graphs, which is an online version of $\lambda$-choosability. Let $\psi(\lambda)$ be the minimum number of vertices in a non-$\lambda$-paintable $k_{\lambda}$-chromatic graph. We determine the value of $\psi(\lambda)$ when each integer in $\lambda$ is at most $2$.

\medskip

\end{abstract}

\section{Induction}

For a graph $G$, a \emph{proper $k$-colouring} of $G$ is a mapping $f:V(G) \rightarrow \{1,2,\ldots,k\}$ such that $f(u) \neq f(v)$ for any edge $uv$ of $E(G)$. We say $G$ is \emph{$k$-colourable} if $G$ has a proper $k$-colouring. The \emph{chromatic number} $\chi(G)$ of $G$ is the minimum positive integer $k$ such that $G$ is $k$-colourable. A \emph{list assignment} of a graph $G$ is a mapping $L$ which assigns to each vertex $v$ of $G$ a set $L(v)$ of colours.
The {\em pallete} of a list assignment $L$ is $p(L) = \bigcup_{v \in V(G)} L(v)$. A \emph{proper $L$-colouring} of $G$ is a proper colouring $f$ of $G$ such that $f(v)\in L(v) $ for each vertex $v$ of $G$. We say $G$ is \emph{$L$-colourable} if $G$ has a proper $L$-colouring. A {\em $k$-list assignment} of $G$ is a list assignment $L$ of $G$ such that   $|L(v)| \geqslant k$ for every vertex $v$. If $G$ is $L$-colourable for every $k$-list assignment $L$ of $G$, then we say that $G$ is \emph{$k$-choosable}. The \emph{choice number or list chromatic number} $ch(G)$ of $G$ is the minimum positive integer $k$ such that $G$ is $k$-choosable.

It follows from the definitions that $\chi(G) \le ch(G)$ for any graph $G$, and it was shown in \cite{ERT1980} that bipartite graphs can have arbitrarily large choice number. An interesting problem is for which graphs $G$, $\chi(G)=ch(G)$. Such graphs are called {\em chromatic choosable}.
Chromatic choosable graphs has been studied extensively in the literature. There are a few challenging conjectures that assert certain families of graphs are chromatic choosable. One central problem in this area is the \emph{list colouring conjecture}, which asserts that line graphs are chromatic choosable \cite{BH1985}.
Another well-known problem in this area is Ohba conjecture, which has been proved by Noel, Reed and Wu in 2014 \cite{NRW2015}. It asserts that every graph $G$ with $|V(G)| \le 2\chi(G)+1$ is chromatic choosable.
For a positive integer $k$, let $\phi(k)$ be the minimum integer $n$ such that there exists a non-$k$-choosable $k$-chromatic graph with $n$ vertices. The result of Noel, Reed and Wu asserts that $\phi(k) \ge 2k+2$.
This result is sharp when $k$ is even, i.e. if $k$ is even then $\phi(k)=2k+2$. If $k$ is odd, then it is known that $\phi(k) \le 2k+3$ and it was a conjectured by Noel \cite{JA2013} that $\phi(k)=2k+3$.
In this paper, we are also interested in the minimum number of vertices of a $k$-chromatic graph $G$ which is not $L$-colourable for some $k$-list assignment $L$, however, we put some restrictions on the list assignments in consideration.

Assume $\lambda=\{k_1,k_2, \ldots, k_q\}$ is a multi-set of positive integers. Let $k_{\lambda} = \sum_{i=1}^q k_i$. A {\em $\lambda$-list assignment} of $G$ is a   $k_{\lambda}$-list assignment $L$ of $G$ of the form $L = \bigcup_{i=1}^qL_i$ (i.e., $L(v) = \bigcup_{i=1}^qL_i(v)$ for each vertex $v$), where each $L_i$ is a $k_i$-list assignment of $G$ and 
$p(L_i) \cap p(L_j) = \emptyset$ for $i \ne j$.  In other words,   the colour set $\bigcup_{v\in V(G)}L(v)$ can be partitioned into $q$ subsets $C_{1}, C_{2}, \ldots, C_{q}$ with $|L(v)\cap C_{i}| \ge  k_{i}$ for each vertex $v$ of $G$. We say $G$ is \emph{$\lambda$-choosable} if $G$ is $L$-colourable for any $\lambda$-list assignment $L$ of $G$. The concept of \emph{$\lambda$-choosability} can be viewed as a refinement of choosability of graphs. If $|\lambda|=1$, i.e. $\lambda=\{k_{\lambda} \}$, then $\lambda$-choosability is the same as $k_{\lambda}$-choosability; if $|\lambda| =k_{\lambda}$, i.e.
$\lambda=\{1,1,\ldots, 1\}$ ($k_{\lambda}$ copies of $1$), then
$\lambda$-choosability is the equivalent to $k_{\lambda}$-{colourability}. 
Given a multiset $\lambda$ and a positive integer $a$, let $m_{\lambda}(a) \ge 0$ be the multiplicity of $a$ in $\lambda$.
So  {$k_{\lambda} = \sum_{a \in \mathbb{N}} a m_{\lambda}(a)$}. 
We say multisets $\lambda_1, \lambda_2, \ldots, \lambda_t$ of positive integers form a partition of $\lambda$ if for each positive integer $a$, 
$m_{\lambda}(a) = \sum_{i=1}^t m_{\lambda_i}(a)$.
For $\lambda = \{k_1,k_2,\ldots, k_q\}$ and $\lambda'= \{k'_1,k'_2,\ldots,k'_p\}$, we write $\lambda \le \lambda'$ if there is a partition $\lambda'_1, \lambda'_2,\ldots, \lambda'_q$ of $\lambda'$ such that for $i=1,2,\ldots, q$, $k_i \le k_{\lambda'_i}$. Then $\le$ is a partial order on all multisets of positive integers. 
For example, there are   four partitions of $4$: $\{4\}, \{1,3\}, \{2,2\}, \{1,1,2\}, \{1,1,1,1\}$. We have $\{4\} \le \{1,3\} \le \{1,1,2\} \le \{1,1,1,1\}$ and also $\{4\} \le \{2,2\} \le \{1,1,2\}$. However, $\{1,3\}$ and $\{2,2\}$ are incomparable. 

It follows from the definitions that if $\lambda \le \lambda'$, then every $\lambda$-choosable graph is $\lambda'$-choosable. Conversely, it was proved in \cite{zhu2019} that if $\lambda \not\le \lambda'$, then there is a $\lambda$-choosable graph which is not $\lambda'$-choosable. 

All the partitions $\lambda$ of $k_{\lambda}$ are sandwiched between $\{k_{\lambda}\}$ and $\{1,1,\ldots, 1\}$. As observed above, {$\{k_{\lambda}\}$}-choosability is the same as $k_{\lambda}$-choosability, and $\{1,1,\ldots, 1\}$-choosability is equivalent to $k_{\lambda}$-colourability. For other partitions $\lambda$ of $k_{\lambda}$,  $\lambda$-choosability reveals a complex hierarchy of colourability of graphs.

The concept of $\lambda$-choosability is introduced in \cite{zhu2019} very recently. Nevertheless, some problems studied in the literature can be expressed in this language. For planar graphs, it follows from the classical result  Voigt \cite{Voigt1993} and the famous Four Colour Theorem that every planar graph is $\{1,1,1,1\}$-choosable but there are planar graphs that are not $\{4\}$-choosable. As a  strengthening of the result in \cite{Voigt1993},
it was proved by Choi and Kwon \cite{CK2017} that there are planar graphs that are not $\{1,3\}$-choosable, and the asked if every planar graph is $\{1,1,2\}$-choosable. Very recently, this question is answered in negative by Kemnitz and Voigt   \cite{KV2018} who proved that there are planar graphs that are not $\{1,1,2\}$-choosable.  

In this paper, we are interested in the smallest number of vertices in a  non-$\lambda$-choosable $k_{\lambda}$-chromatic graph.

 \begin{definition}
 	Assume $\lambda = \{k_1, k_2, \ldots, k_q\}$ is a multi-set of positive integers.
 	Let $$\phi(\lambda)=\min\{n: \text{ there exists a non-$\lambda$-choosable $k_{\lambda}$-chromatic $n$-vertex graph}\}.$$
 \end{definition}

The problem is to determine or find bounds for $\phi(\lambda)$.
If $\lambda'$ is obtained from $\lambda$ by replacing some integers in $\lambda$ by partitions of these integers, then we say $\lambda'$ is a refinement of $\lambda$.
If $\lambda'$ is a refinement of $\lambda$, then  $\lambda \le \lambda'$ and $k_{\lambda} = k_{\lambda'}$.
Therefore we have following proposition.

\begin{proposition}
	\label{prop-easy}
	If $\lambda'$ is a refinement of $\lambda$, then $\phi(\lambda') \ge \phi(\lambda)$. In particular,
	as any multi-set $\lambda$ is a refinement of $\{k_{\lambda}\}$, we conclude that
	$\phi(\lambda) \ge \phi(k_{\lambda}) \ge 2k_{\lambda} +2$.
\end{proposition}

If $|\lambda| = k_{\lambda}$, i.e., $\lambda$ consists of $k_{\lambda}$ copies of $1$'s, then every $k_{\lambda}$-chromatic graph is $\lambda$-choosable. In this case, we set $\phi(\lambda) = \infty$. 
In the following, we only consider multisets $\lambda$ of positive integers with $|\lambda| < k_{\lambda}$. 

For such multisets of positive integers, it is natural that $\phi(\lambda)$ also highly depends to the number of $1$'s in $\lambda$.
  Let $1_{\lambda}$ be the multiplicity of $1$ in $\lambda$ (i.e. number of copies of $1$ in $\lambda$),
  and let $o_{\lambda}$ be the number of integers in $\lambda$ that are odd.
 In this paper, we prove the following result.

\begin{theorem}
	\label{range}
    For any multi-set $\lambda$ of positive integers with $|\lambda| < k_{\lambda}$, $$2k_{\lambda}+1_{\lambda}+2 \leqslant \phi(\lambda ) \leqslant 2k_{\lambda}+o_{\lambda}+2.$$
    In particular, if $1_{\lambda} =o_{\lambda}=t$, then $\phi(\lambda)=2k_{\lambda}+t+2.$
\end{theorem}

The gap between the lower bound and upper bound for $\phi(\lambda)$ is $o_{\lambda}- 1_{\lambda}$.
The precise value of $\phi(\lambda)$ seems to be a difficult problem. In particular, if
$\lambda$ consists of a single odd integer $k$, we only know that $2k+2 \le \phi(\lambda) \le 2k+3$.
So there is a gap of $1$.
For general   multi-set $\lambda$ of positive integer,
we prove the following result which suggests that, among the two variants $1_{\lambda}$ and $o_{\lambda}$,  $1_{\lambda}$ plays a more important role in determining $\phi(\lambda)$.

\begin{theorem}
	\label{thm-no1}
For any multiset $\lambda$ of positive integers with $|\lambda| < k_{\lambda}$, $\phi(\lambda) \le 2k_{\lambda}+51_{\lambda}+3$.	In particular, if $1_{\lambda}=0$, then $2k_{\lambda}+2 \leqslant \phi(\lambda) \leqslant 2k_{\lambda}+3$. 
\end{theorem}

The paintability of a graph is an online version of list colourability  \cite{S2009,Zhu2009}.
Assume $f:V(G)\longrightarrow\{0,1,2,\ldots\}$ is a mapping which assigns a non-negative integer $f(v)$ to each vertex $v$.
The \emph{$f$-painting game} on $G$ is played by two players: Lister and Painter.
Initially, each vertex $v$ of $V(G)$ has $f(v)$ tokens, and is uncoloured.
In each round, Lister chooses a nonempty subset $U$ of uncoloured vertices, takes away one token from each $v\in U$. Painter colours vertices in an independent set $I$ of $G$ contained in $U$.
If at the end of a certain round, there is an uncoloured vertex with no tokens left, then Lister wins. Otherwise all the vertices are eventually coloured, then Painter wins.
We say $G$ is \emph{ $f$-paintable} if painter has a winning strategy in the $f$-painting game on $G$. We say $G$ is \emph{ $k$-paintable} when $G$ is $f$-paintable for the constant mapping $f(v)= k$ for all $v \in V(G)$. The \emph{painter number} of $G$, denoted by  $\chi_{p}(G)$ is defined as
$$\chi_{p}(G)=min\{ k: G \text{ is $k$-paintable} \}.$$
It follows from the definition that $ch(G) \le \chi_P(G)$ for every graph $G$ \cite{Zhu2009},
and it is known that the gap $\chi_P(G)-ch(G)$ can be arbitrarily large \cite{DL2016}.
An on-line version of Ohba's conjecture is proposed in \cite{KKLZ2012}: if $|V(G)|\leqslant2\chi(G)$, then $\chi(G)=\chi_{p}(G)$. We define $\psi(k)$ to be the minimum number of vertices in a non-$k$-paintable
$k$-chromatic graph. The above conjecture asserts that $\psi(k) \ge 2k+1$.
It was proved in \cite{HWZ2012} that  $\psi(k) \le 2k+1$, i.e. for each integer $k\ge 3$, there exsits a $k$-chromatic graph
$G$ with  $|V(G)|= 2k +1$ for which $\chi_{p}(G) > k$. So the conjecture, if true, is sharp. 
Some special cases of
the conjecture are verified (cf. \cite{KMZ2014}), however, the conjecture itself is largely open.
We extend the definition of $k$-paintability to $\lambda$-paintability as follows:

\begin{definition}
	Assume $\lambda =\{k_{1},k_{2},\ldots,k_{q}\}$ is a multi-set of positive integers.
	A $\lambda$-painting game on $G$
	is a $k$-painting game on $G$ for which the following holds: There are integers $k'_1, k'_2, \ldots, k'_q$
	such that in the first $k'_1$ rounds, each vertex $v$ of $G$ is selected by Lister in $k_1$ rounds;
	in the next $k'_2$ rounds, each vertex $v$ of $G$ is selected by Lister in $k_2$ rounds; and so on.
	We say $G$ is \emph{$\lambda$-paintable} if painter has a winning strategy in the $\lambda$-painting game on $G$.
\end{definition}

 For a multi-set $\lambda=\{k_1, k_2, \ldots, k_q\}$ of positive integers, let $\psi(\lambda)$ be the least integer $n$ such that there exists a non-$\lambda$-paintable $k_{\lambda}$-chromatic $n$-vertex graph. Again if
 $|\lambda|=k_{\lambda}$, i.e. $\lambda = \{1,1,\ldots, 1\}$, then $\lambda$-paintable is equivalent to
 $k_{\lambda}$-colourable. In this case, we set $\psi(\lambda)=\infty$.
 We prove the following result.

\begin{theorem}
	\label{thm-paint}
	Assume $\lambda=\{k_1, k_2, \ldots, k_q\}$ is a non-trivial multi-set of positive integer. If  $k_i \le 2$ for each $i$, then $\psi(\lambda)=2k_{\lambda}+1_{\lambda}+2$.
\end{theorem}
	
	The union $\lambda \cup \lambda'$ of two multisets $\lambda$ and $\lambda'$ is obtained by adding the multiplicity of each integer in $\lambda$ and $\lambda'$. For example, if $\lambda = \{1,1,2\}$ and $\lambda'=\{1,2,3\}$, then $\lambda \cup \lambda'= \{1,1,1,2,2,3\}$. 
Equivalently, a multiset $\lambda$ of positive integers can also be denoted by a mapping
$f_{\lambda}: \{1,2,\ldots, \} \to \{0,1,\ldots,\}$, where $f_{\lambda}(i)$ is the multiplicity of $i$ in $\lambda$, i.e. $\lambda$ contains $f_{\lambda}(i)$ copies of $i$. By this definition, 
$f_{\lambda \cup \lambda'} = f_{\lambda} + f_{\lambda'}$.

\section{Proof of Theorem \ref{range}}
Assume $0 \leq t \prec k$ are integers. We denote by $\lambda_{k,t}$ the multi-set consisting $t$ copies of 1 and one copy of $k-t$.
 We denote $ch_{t}(G)$ by the minimum positive integer $k$ such that $G$ is $\lambda_{k,t}$-choosable.
Equivalently, $ch_t(G)$ is the minimum integer $k$ such that for any $k$-list assignment $L$ of $G$ for which
$|\bigcap_{v \in V(G)} L(v)| \ge t$, $G$ is $L$-colourable. The parameter $ch_t(G)$ is first studied in \cite{CK2017},
and is called the {\em $t$-common choice number} of $G$. It was proved in \cite{CK2017} that if $G$ is a $k$-chromatic graph, then
\begin{displaymath}
\chi(G)=ch_{k-1}(G)\leqslant ch_{k-2}(G)\leqslant\ldots\leqslant ch_{1}(G)\leqslant ch(G).
\end{displaymath}

Now we give a lower bound for $\phi(\lambda_{k,t})$.

\begin{theorem}
\label{lower lemma}
For any integer $0 \le t < k$, $\phi(\lambda_{k,t}) \ge 2k+t+2$. In other words,
  any $k$-chromatic graph $G$ with $|V(G)| \leq 2k+t+1$  is $\lambda_{k,t}$-choosable.
\end{theorem}
\begin{proof}
We prove Theorem \ref{lower lemma} by induction on $t$.
If $t=0$, then the condition follows from Theorem \cite{NRW2015}.

Assume $t \geq 1$, and the theorem holds for $t-1$.
Let $G$ be a $k$-chromatic graph with $|V(G)| \leq 2k+t+1$.
If $|V(G)| \leq 2k+t$, then by induction hypothesis, $G$ is $\lambda_{k,t-1}$-choosable. Since $\lambda_{k,t}$ is a refinement of $\lambda_{k,t-1}$, by Proposition \ref{prop-easy}, $G$ is $\lambda_{k,t}$-choosable.

Thus  we may assume that $|V(G)|=2k+t+1$.
Let $c$ be a $k$-colouring of $G$, and let $X$ be a largest colour class. Then $|X| \geq 3$, and $G-X$ is a $(k-1)$-chromatic graph. Since $|V(G)-X| \le 2(k-1)+(t-1)+1$,
by induction hypothesis,  $G-X$ is $\lambda_{k-1,t-1}$-choosable.

Assume $L$ is a $\lambda_{k,t}$-list assignment of $G$. It follows from the definition that there is a set $C$ of colours such that for each vertex $v$ of $G$, $|L(v) \cap C|=1$. Let $L'(v)=L(v)-C$ for each vertex $v$ of $G$.  Then $L'$ is a $\lambda_{k-1,t-1}$-list assignment of $G$. Since $G-X$ is  $\lambda_{k-1,t-1}$-choosable,
there is an $L'$-colouring $f$ of $G-X$. By colouring each vertex $v \in X$ by the colour in $L(v) \cap C$, we obtain an extension of $f$ that is an $L$-colouring of $G$. This proves that $G$ is $\lambda_{k,t}$-choosable.
\end{proof}

Next, we present an upper bound on $\phi(\lambda)$.
First we consider the case that $\lambda$ consists of integers at most $2$, i.e. $\lambda=\{1,1,\ldots, 1, 2, 2, \ldots, 2\}$.

\begin{lemma}
	\label{upper lemma 1}
	Assume $\lambda=\{1,1,\ldots,1,2,\ldots,2\}$ where 1 has multiplicity $a$ and 2 has multiplicity $b \geq 1$. Let $k=k_\lambda=2b+a$. Then there exists a $k$-chromatic graph $G$ with $|V(G)|=2k+a+2$ which is not $\lambda$ -choosable.
\end{lemma}
\begin{proof}
	Let $G=K_{3*(a+b+1),1*(b-1)}$ be the $k$-partite graph with $a+b+1$ partite sets of size 3 and $b-1$ partite sets of size 1. The $a+b+1$ partite set of size 3 are $\{u_{i1}, u_{i2}, u_{i3}\}$ for $i=1, 2, \ldots, a+b+1$ and the $b-1$ partite sets of size 1 are $\{v_{j}\}$ for $j=1, 2, \ldots, b-1$. Then $G$ is a $k$-chromatic graph with $|V(G)|=2k+a+2$.
	
	We will construct a $\lambda$-list assignment $L$ such that $G$ is not $L$-colourable.
	
	\begin{itemize}
		\item  For $i=1,2,\ldots, b$, let $A_{i}$ be disjoint sets of size $3$,
		\item  For $i=1, 2, \ldots, b$,  let $A_{i,1}, A_{i,2}, A_{i,3}$ be the three $2$-subsets of $A_i$.
		Thus 
		$A_{i,1}\cup A_{i,2}\cup A_{i,3}=A_{i}$ and $A_{i,1}\cap A_{i,2}\cap A_{i,3}=\emptyset$,
		\item 	Let $E$ be a set of $a$ colours, and the sets $E, A_{i}$ are disjoint,
		\item   For $j=1,2,3$, let $B_j=\bigcup_{i=1}^b A_{i,j} \cup E$.
	\end{itemize}

	Let $L$ is an $\lambda$-list assignment of $G$ defined as follows:
	\begin{enumerate}
		\item   $L(u_{i,j})=B_j$, $i=1,2,\ldots,a+b+1$, $j=1,2,3$.
		\item $L(v_{j})=B_{1}$, $j=1,2,\ldots,b-1$.
	\end{enumerate}
	
		We may assume that $E=\{1, 2, \ldots, a\}$.
		For each vertex $v$ of $G$, the following hold:
	for each colour $c \in E$, $|L(v) \cap \{c\}|=1$; for each $i=1,2,\ldots, b$,
	$|L(v) \cap A_i| =2$.  So $L$ is a $\lambda$-list assignment, where there are
	$a$ copies of $1$ and $b$ copies of $2$.

	Now we show that $G$ is not $L$-colourable.
	Assume to the contrary that   $G$ has a proper $L$-colouring $\varphi$.
	Let $$X = \{u_{i,1}: i=1,2,\ldots, a+b+1\} \cup \{v_j: j=1,2,\ldots, b-1\}.$$
	Then $G[X]$ is a copy of $K_k$. As $L(v) = B_1$ for each vertex $v \in X$, we conclude that
	$\varphi(X)=B_{1}$.

	For $i=1,2,\ldots, a+b+1$ and $j=2,3$,
	let 
	$$L'(u_{i,j}) = (L(u_{i,j}) - B_1) \cup (\{\varphi(u_{i,1})\} \cap L(u_{i,j})).$$ 
	As $u_{i,j}$ is adjacent to every vertex of $X$ except 
	$u_{i,1}$, we conclude that 
	$\varphi(u_{i,j}) \in L'(u_{i,j})$.

If $\varphi(u_{i,1})\in  { \bigcup _{i=1}^b A_{i,1}}$, then $\varphi(u_{i,1})$ is contained in exactly one of $L(u_{i,2}), L(u_{i,3})$.
(As each colour in $A_i$ belongs to exactly two of the lists $L(u_{i,1}), L(u_{i,2}), L(u_{i,3})$).
 Thus 
  $$\{L'(u_{i,2}),L'(u_{i,3})\}=\{(\bigcup_{i=1}^{b}A_{i})-B_{1},((\bigcup_{i=1}^{b}A_{i})-B_{1})\cup \{\varphi(u_{i,1})\}\},$$

If $\varphi(u_{i,1}) \in E$, then  $\varphi(u_{i,1})$ is contained in both $L(u_{i,2}), L(u_{i,3})$. 
Thus
$$\{L'(u_{i,2}),L'(u_{i,3})\}=\{((\bigcup_{i=1}^{b}A_{i})-B_{1})\cup \{\varphi(u_{i,1})\}\}.$$

Assume   $m$ vertices  in  $\{ u_{i,1}: i=1,2,\ldots, a+b+1\}$ are coloured by a colour of $E$. Then $m \le |E| =a$.
Let $Y=\{u_{i,j}: L'(u_{i,j})=(\bigcup_{i=1}^{b}A_{i})-B_{1}\}$. Then $|Y|=a+b+1-m \geq a+b+1-a=b+1$. Thus  $G[Y]$ contains a copy of $K_{b+1}$, and all the vertices in $Y$ are coloured by colours from $(\bigcup_{i=1}^{b}A_{i})-B_{1}$. However, $|(\bigcup_{i=1}^{b}A_{i})-B_{1}|=b$. This is a contradiction.
\end{proof}

\begin{theorem}
\label{upper bound 1}
For any non-trivial multi-set $\lambda$ of positive integers, $\phi(\lambda) \leq 2k_\lambda+o_\lambda+2.$
\end{theorem}
\begin{proof} For any non-trivial multiset $\lambda$ of positive integers, let  
	$\lambda'$ be the multiset which consists of $o_{\lambda}$ copies of $1$, and
	$(k_{\lambda}-o_{\lambda})/2$ copies of $2$. Then $\lambda'$ is a refinement of $\lambda$.
	 By Proposition \ref{prop-easy} and Lemma \ref{upper lemma 1}, $\phi(\lambda) \le \phi(\lambda')= 2k_{\lambda'}+1_{\lambda'}+2 =  2k_\lambda+o_\lambda+2$.
\end{proof}

\begin{corollary}
If $\lambda$ is a non-trivial multiset of positive integers and  $o_{\lambda}=1_{\lambda}$,   then $\phi(\lambda)=2k_{\lambda}+o_{\lambda}+2$.
\end{corollary}

\begin{theorem}
	\label{special lemma}
	Assume  $\lambda$ is a non-trivial multiset  of positive integers. 
	Then there is a $k_\lambda$-chromatic graph $G$ with   $|V(G)|= 2k_\lambda + 5 1_{\lambda} + 3$ such that $G$ is not $\lambda$-choosable. Hence $\phi(\lambda)\le 2k_{\lambda}+ 51_{\lambda} + 3$.
\end{theorem}
\begin{proof}
Assume $\lambda =\{k_1,k_2, \ldots, k_q\}$, where $k_1, k_2, \ldots, k_a$ are odd, and 
$k_{a+1}, k_{a+2}, \ldots, k_q$ are even. 

	\bigskip
	\noindent
	{\bf Case 1}  $1_{\lambda} = 0$.   

	Let $k=k_{\lambda}$ and let $G=K_{1*5,(k -1)*2}$ be the complete $k $-partite graph with one partite set of size $5$ and $k  -1$ partite
	sets of size $2$.
	The  $k-1$ partite sets of size $2$ are $V_i=\{u_{i,1}, u_{i,2}\}$ for $i=1,2,\ldots, k-1$,
	and the partite set of size $5$ is $V_k= \{v_1, v_2, v_3, v_4, v_5\}$. Then $G$ is a $k $-chromatic graph with $|V(G)|=2k+3$.
	We shall construct a $\lambda$-list assignment $L$ of $G$ so that
	$G$ is not $L$-colourable.

	\begin{itemize}
		\item For $i=1,2,\ldots,a$, $j=1,2,3,4$, let $S_{i,j}$ be disjoint sets of size $\frac{k_i-1}{2}$.
		\item  For $i=a +1, a +2, \ldots, q$, $j=1,2,3,4$, let $T_{i,j}$ be disjoint sets of size $\frac{k_i}{2}$.
		\item 	For $j=1,2,3,4$, let
		$$S_j = \bigcup_{i=1}^{a} S_{i,j}, \ T_j =  \bigcup_{i=a +1}^q T_{i,j}.$$
		\item 	Let $E$ be a set of $a$ colours, and the sets $E, S_{i,j}, T_{i,j}$ are pairwise disjoint.
		\item   For $i=1,2,\ldots, a$, let $X_i$ be an arbitrary subset of $\bigcup_{j=1}^4 S_{i,j}$ of size $k_i$, and let $X= \bigcup_{i=1}^{a} X_i$. (Note that here we used the fact that $k_i \ne 1$, for otherwise $S_{i,j}$ would be emptyset and such a set $X_i$ does not exist.)
	\end{itemize}

	Let $L$ be the $\lambda$-list assignment of $G$ defined as follows:
	\begin{enumerate}
		\item   $L(v_1)=E\cup S_1\cup S_3\cup T_1\cup T_3$, $L(v_2)=E\cup S_1\cup S_4\cup T_1\cup T_4$, \\
	$L(v_3)=E\cup S_2\cup S_3\cup T_2\cup T_3$, $L(v_4)=E\cup S_2\cup S_4\cup T_2\cup T_4$,\\
 $L(v_5) = X \cup T_2\cup T_4$;
		\item $L(u_{i,1})=E\cup S_1 \cup S_2 \cup T_1 \cup T_2$,  for $i=1,2,\ldots,k -1$;
		\item   $L(u_{i,2})=E\cup S_3 \cup S_4 \cup T_3 \cup T_4$,  for $i=1,2,\ldots,k -1$.
	\end{enumerate}
	
	First we verify that $L$ is a $\lambda$-list assignment of $G$.
	We may assume that $E=\{1,2,\ldots, a\}$. Let
\[
C_i = \begin{cases}
\{i\} \cup  \bigcup_{j=1}^4 S_{i,j}, & \text{ for $i=1, 2, \ldots,  a$ }, \cr
\bigcup_{j=1}^4 T_{i,j}, &  \text{ for $i=a +1, a +2, \ldots, q$ }.
\end{cases}
\]
	
	Then $C_1, C_2, \ldots, C_q$ are pairwise disjoint and for each vertex $v$ of $G$,
	$|L(v) \cap C_i| = k_i$ for $i=1,2,\ldots, q$.
		So $L$ is a $\lambda$-list assignment of $G$.
	
	Next we show that $G$ is not $L$-colourable.
	Assume to the contrary that $\varphi$ is an $L$-colouring of $G$.  For $j=1,2$, let $U_j = \{u_{1,j}, u_{2,j}, \ldots, u_{k -1, j}\}$. Each of $U_1, U_2$ induces a complete graph of order $k -1$. Hence
	$|\varphi(U_j)| = k -1$. As $$\varphi(U_1) \subseteq E\cup S_1 \cup S_2 \cup T_1 \cup T_2$$
	and $$\varphi(U_2) \subseteq E \cup  S_3 \cup S_4 \cup T_3 \cup T_4,$$ and
	each of $E\cup S_1 \cup S_2 \cup T_1 \cup T_2$ and $E \cup  S_3 \cup S_4 \cup T_3 \cup T_4$ is a set of size $k$,
	we conclude that
	$$|E\cup S_1 \cup S_2 \cup T_1 \cup T_2 - \varphi(U_1)|=1   \text{ and } |E \cup  S_3 \cup S_4 \cup T_3 \cup T_4 - \varphi(U_2)|=1. $$
	Assume
	$$E\cup S_1 \cup S_2 \cup T_1 \cup T_2 - \varphi(U_1) = \{c_1\}   \text{ and } E \cup  S_3 \cup S_4 \cup T_3 \cup T_4 - \varphi(U_2) =\{c_2\}.$$

	For each vertex $v \in V_k $, $v$ is adjacent to all the vertices in $U_1 \cup U_2$. Hence 
	$\varphi(v) \in L(v)  \cap \{c_1, c_2\}$. In particular,
	$L(v) \cap \{c_1, c_2\}   \ne \emptyset$.
	As $L(v_5) \cap E = \emptyset$, we conclude that $\{c_1, c_2\} - E \ne \emptyset$.

	First we assume that $c_1 \in E$.  Then $c_2 \notin E$. This implies that   $$E \subseteq \varphi(U_2).$$
	As all the vertices $v  \in V_k$ are adjacent to every vertex in $U_2$, we conclude that $\varphi(v) \notin E$
	and hence 
	   $\varphi(v)= c_2$.
 But $$\bigcap_{v \in V_k} L(v)= \emptyset,$$
 which is a contradiction.
	
	The case that $c_2 \in E$ is symmtric.
	
	Assume $c_1, c_2\notin E$. Thus $c_1 \in S_1 \cup S_2 \cup T_1 \cup T_2$ and $c_2 \in S_3 \cup S_4 \cup T_3 \cup T_4$.
	
	If $c_1 \in S_1 \cup T_1$, then since $L(v_3) = E \cup S_2 \cup S_3 \cup T_2 \cup T_3$,
	 we must have $c_2 \in S_3 \cup T_3$ (for otherwise, $L(v_3) \cap \{c_1, c_2\} = \emptyset$.
	But then $L(v_4) \cap \{c_1, c_2\} = \emptyset$. This is a contradiction.
	
	If $c_1 \in S_2 \cup T_2$, then since  $L(v_1) = E \cup S_1 \cup T_3 \cup T_1 \cup T_3$,
	we must have $c_2 \in S_3 \cup T_3$. But then
	and $L(v_2) \cap \{c_1, c_2\} = \emptyset$, again a contradiction.
	
	\bigskip
	\noindent
	{\bf Case 2}  $1_{\lambda} > 0$. 
	
	Assume  $\lambda =\{k_1,k_2, \ldots, k_q,1,1,\ldots, 1\}$, where $k_1, k_2, \ldots, k_a$ are odd integer greater than 1,  
	$k_{a+1}, k_{a+2}, \ldots, k_q$ are even, and there are $t >0$ copies of $1$'s.
	
	Let $k=k_{\lambda}-t$ and let 
	$G=K_{1*5,(k -1)*2}$ be the complete
		 $k$-partite graph as in Case 1,
		 and let $G'$ be obtained from $G$ by adding  $t$ partite sets of size $7$.
	Vertices in $G$ are named as in Case 1, and 
	let $\{x_{i,1}, x_{i,2}, \ldots, x_{i,7}\}$ ($i=k+1,k+2,\ldots,k+t$) be the partite sets of size $7$.
	
	For vertices in $G$, let $L(u_{i,j})$ and $L(v_i)$ be the lists as in Case 1.
	For $j=1,2,\ldots, 5$, let $L(x_{i,j})=L(v_j)$ and let  {$L(x_{i,6}) = L(u_{1,1})$ and $L(x_{i,7})=L(u_{1,2})$}. 
	Let $F$ be a set of $t$ new colours, and for each vertex $v$ of  {$G'$}, let $L'(v)=L(v) \cup F$. 
	It is routine to check that $L'$ is $\lambda$-list assignment of $G'$. 
	
	Now we show that $G'$ is not $L'$-colourable. Assume to the contrary that 
$\phi$ is an $L'$-colouring of $G'$. The $t$ colours in $F$ are used to colour $t$ partite sets. 
If the $t$ partite sets coloured by colours in $F$ are 
the $t$ partite sets of size $7$, then the restriction of $\phi$ to $G$ is an $L$-colouring of $G$, which is contrary to Case 1. If $s$ colours from $F$ are used to colour partite sets of size $2$ or the partite set of size $5$,
then there are $s$ partite sets of size $7$ left uncoloured. We delete vertices $x_{i,1},x_{i,2}, x_{i,3}, x_{i,4},x_{i,5}$ from a partite set of size $7$, the remaining two vertices play the role of a partite set of size $2$ in $G$. We delete vertices $x_{i,6},x_{i,7}$ from a partite set of size $7$, the remaining five vertices play the role of a partite set of size $5$ in $G$. So the restriction of $\phi$ to the remaining vertices also gives an $L$-colouring of $G$, which is contrary to Case 1. 
\end{proof}

\section{Proof of Theorem \ref{thm-paint}}

This section proves 
 Theorem \ref{thm-paint}, which states   that for non-trivial multi-set $\lambda=\{k_1, k_2, \ldots, k_q\}$ of positive integer, if  $k_i \le 2$ for each $i$, then $\psi(\lambda)=2k_{\lambda}+1_{\lambda}+2$.

\begin{proposition}
	\label{partition}
	If $V(G)=V_1 \cup V_2$ is a partition of $V(G)$, and $G[V_i]$ is $\lambda_i$-paintable, then $G$ is $(\lambda_1 \cup \lambda_2)$-paintable.
\end{proposition}
\begin{proof}
By the definition of $\lambda$-painting game, we may assume that  
 there are integers $k'_1, k'_2$
such that in the first $k'_1$ rounds, Lister plays the $\lambda_1$-painting game on $G$.
Since $G[V_1]$ is $\lambda_1$-paintable, Painter has a strategy to colour all vertices in 
$V_1$ in the first $k'_1$ rounds. Similarly, Painter has a strategy to colour all vertices in $V_2$ in the 
remaining rounds. 
\end{proof}

\begin{lemma}
	\label{2,2,...,2-paintable}
If $G$ is a $2q$-chromatic graph with $|V(G)| \leqslant 4q+1$ and $\lambda$ consists of   $q$ copies of $2$, then $G$ is $\lambda$-paintable. 
\end{lemma}                                                  
\begin{proof}
	We prove the lemma by induction on $\chi(G)$.
	As the star $K_{1,m}$ and complete bipartite graph $ K_{2,3}$ are $2$-paintable,
	 the lemma is true if $\chi(G)=2$.
Assume $k=\chi(G)=2q \ge 4$ and $V_1, V_2, \ldots, V_k$ are the colour classes, with 
$|V_1| \le |V_2| \le \ldots \le |V_k|$. If $|V_k| \ge 4$, then since $|V(G)| \le 2k+1$, 
we conclude that $|V_1|=1$. Let $X_1=V_1 \cup V_k$ and $X_2=V_2\cup V_3 \cup \ldots \cup V_{k-1}$.

Then $G[X_1]$ is a star and hence is $2$-paintable, and $G[X_2]$ is a $(k-2)$-chromatic graph with 
at most $2(k-2)$ vertices, and hence by induction hypothesis $G[X_2]$ is $\lambda'$-paintable,
where $\lambda'$ consists of $(q-1)$ copies of $2$. 
 
Otherwise, $|V_k|=3$, then $|V_1| \leqslant 2$. Let $X'_1=V_1 \cup V_k$ and $X'_2=V(G)-X'_1$. Then $G[X'_1]$ is a copy of $K_{2,3}$ or $K_{1,3}$ and hence is $2$-paintable. $G[X'_2]$ is a $(k-2)$-chromatic graph with at most $2(k-2)+1$ vertices, and then $G[X'_2]$ is $\lambda'$-chromatic, where $\lambda'$ consists of $(q-1)$ copies of $2$. 
 
By Lemma \ref{partition}, $G$ is $\lambda = \lambda' \cup \{2\}$-paintable.

\end{proof}

\noindent
{\bf Proof of   Theorem \ref{thm-paint}}

It follows from the definition that for any multiset $\lambda$ of positive integers, every $\lambda$-paintable graph is $\lambda$-choosable.  
Hence by Theorem \ref{upper lemma 1},  {$\psi(\lambda) \leqslant \phi(\lambda) \leqslant 2k_\lambda+1_\lambda+2$}. 
 
It remains to show that if every integer in $\lambda$ is at most $2$, and 
  $G$ is a $k_\lambda$-chromatic graph of order at most $2k_\lambda +1_\lambda +1$, then $G$ is $\lambda$-paintable.
 We prove this by induction on $1_\lambda$. 
	
	If $1_{\lambda}=0$, then this is Lemma \ref{2,2,...,2-paintable}.
	Assume $1_{\lambda} \ge 1$ and let $\lambda'$ be obtaind from $\lambda$ by deleting a copy of $1$. Let $c$ be a $k_{\lambda}$-colouring of $G$. Let $X_1$ be a maximum colour class  and $X_2 = V(G)-X_1$. Then 
	$|X_2| \le 2 k_\lambda +1_\lambda +1 -3 = 2(k_{\lambda}-1)+1_{\lambda}
	= 2k_{\lambda'}+1_{\lambda'}+1$. By induction hypothesis, $G[X_2]$ is $\lambda'$-paintable. It is obvious that $G[X_1]$ is $1$-paintable. By Lemma \ref{partition}, $G$ is $\lambda$-paintable.


\begin{thebibliography}{10}
	
	\bibitem{BH1985}
	B.~Bollob\'{a}s and A.~J. Harris.
	\newblock List-colourings of graphs.
	\newblock {\em Graphs Combin.}, 1(2):115--127, 1985.
	
	\bibitem{CK2017}
	Hojin Choi and Young~Soo Kwon.
	\newblock On {$t$}-common list-colorings.
	\newblock {\em Electron. J. Combin.}, 24(3):Paper 3.32, 10, 2017.
	
	\bibitem{DL2016}
	Lech Duraj, Grzegorz Gutowski, and Jakub Kozik.
	\newblock Chip games and paintability.
	\newblock {\em Electron. J. Combin.}, 23(3):Paper 3.3, 12, 2016.
	
	\bibitem{ERT1980}
	Paul Erd\H{o}s, Arthur~L. Rubin, and Herbert Taylor.
	\newblock Choosability in graphs.
	\newblock In {\em Proceedings of the {W}est {C}oast {C}onference on
		{C}ombinatorics, {G}raph {T}heory and {C}omputing ({H}umboldt {S}tate
		{U}niv., {A}rcata, {C}alif., 1979)}, Congress. Numer., XXVI, pages 125--157.
	Utilitas Math., Winnipeg, Man., 1980.
	
	\bibitem{HWZ2012}
	Po-Yi Huang, Tsai-Lien Wong, and Xuding Zhu.
	\newblock Application of polynomial method to on-line list colouring of graphs.
	\newblock {\em European J. Combin.}, 33(5):872--883, 2012.
	
	\bibitem{KV2018}
	Arnfried Kemnitz and Margit Voigt.
	\newblock A note on non-4-list colorable planar graphs.
	\newblock {\em Electron. J. Combin.}, 25(2):Paper 2.46, 5, 2018.
	
	\bibitem{KKLZ2012}
	Seog-Jin Kim, Young~Soo Kwon, Daphne Der-Fen Liu, and Xuding Zhu.
	\newblock On-line list colouring of complete multipartite graphs.
	\newblock {\em Electron. J. Combin.}, 19(1):Paper 41, 13, 2012.
	
	\bibitem{KMZ2014}
	Jakub Kozik, Piotr Micek, and Xuding Zhu.
	\newblock Towards an on-line version of ${O}hba's$ conjecture.
	\newblock {\em European J. Combin.}, 36:110--121, 2014.
	
	\bibitem{JA2013}
	Jonathan~A. Noel.
	\newblock Choosability of graphs with bounded order: Ohba's conjecture and
	beyond.
	\newblock {\em Computer Science}, 2013.
	
	\bibitem{NRW2015}
	Jonathan~A. Noel, Bruce~A. Reed, and Hehui Wu.
	\newblock A proof of a conjecture of {O}hba.
	\newblock {\em J. Graph Theory}, 79(2):86--102, 2015.
	
	\bibitem{S2009}
	Uwe Schauz.
	\newblock Mr. {P}aint and {M}rs. {C}orrect.
	\newblock {\em Electron. J. Combin.}, 16(1):Research Paper 77, 18, 2009.
	
	\bibitem{Voigt1993}
	Margit Voigt.
	\newblock List colourings of planar graphs.
	\newblock {\em Discrete Math.}, 120(1-3):215--219, 1993.
	
	\bibitem{Zhu2009}
	Xuding Zhu.
	\newblock On-line list colouring of graphs.
	\newblock {\em Electron. J. Combin.}, 16(1):Research Paper 127, 16, 2009.
	
	\bibitem{zhu2019}
	Xuding Zhu.
	\newblock A refinement of choosability of graphs.
	\newblock {\em Journal of Combinatorial Theory, Series B}, 2019.
	
\end{thebibliography}
\end{document}